\documentclass[11pt,a4paper]{amsart}

\usepackage[margin=25truemm,top=30truemm,bottom=30truemm]{geometry}
\usepackage[dvipdfmx]{graphicx}
\usepackage{enumerate}
\usepackage{amsthm}
\usepackage{amssymb}
\usepackage{tikz-cd}
\usepackage{here}
\usepackage{mathtools}

\newtheorem{Prop}{Proposition}[section]
\newtheorem{Thm}[Prop]{Theorem}
\newtheorem{Lem}[Prop]{Lemma}
\newtheorem{Rem}[Prop]{Remark}

\theoremstyle{definition}
\newtheorem{Def}[Prop]{Definition}
\newtheorem{Ex}[Prop]{Example}

\makeatletter
\@addtoreset{equation}{section}
\makeatother

\title{Solvable Lie algebras obtained by quivers and Einstein metrics}
\author{Fumika Mizoguchi} 

\date{}
\thanks{This work was supported by JST SPRING, Grant Number JPMJSP2139. This work was partly supported by MEXT Promotion of Distinctive Joint Research Center Program JPMXP0723833165 and Osaka Metropolitan University Strategic Research Promotion Project (Development of International Research Hubs).}

\address[F.~Mizoguchi]{Center for Advanced intelligence Project, RIKEN} 
\email{fumika.mizoguchi@riken.jp}

\DeclareMathOperator{\id}{id}

\DeclareMathOperator{\ad}{ad}
\DeclareMathOperator{\g}{\mathfrak{g}}
\DeclareMathOperator{\n}{\mathfrak{n}}
\DeclareMathOperator{\s}{\mathfrak{s}}
\DeclareMathOperator{\Ric}{\mathrm{Ric}}
\DeclareMathOperator{\ric}{\mathrm{ric}}
\DeclareMathOperator{\Der}{\mathrm{Der}}
\DeclareMathOperator{\Path}{\mathrm{Path}}

\begin{document}
	\maketitle
	\begin{abstract}
		In geometry, it is important to study whether a given Lie group admits a special left-invariant geometric structure.  In our previous work, we constructed nilpotent Lie algebras from finite quivers without cycles by utilizing paths within these quivers. Also, we proved that the simply-connected Lie groups corresponding to these nilpotent Lie algebras always admit left-invariant Ricci solitons. In this paper, we extend our approach by constructing solvable Lie algebras from finite quivers without cycles, adding vertices as paths of length zero. We demonstrate that the simply-connected Lie groups corresponding to these solvable Lie algebras always admit left-invariant Ricci solitons. Moreover, we prove that when the quivers are oriented multi-trees, these Ricci soliton Lie groups are rigid, that is the direct product manifold of a flat manifold and an Einstein manifold. 		
	\end{abstract}
	\section{Introduction}
		In geometry, it is a fundamental problem to study whether a given manifold admits a special metric. Let $(M, g)$ be a Riemannian manifold and $\ric_g$ be the Ricci curvature tensor. The Riemannian metric $g$ is called a {\it Ricci soliton} if there exist $c\in \mathbb{R}$ and a vector field $X$ on $M$ such that 
		\begin{equation}\label{ric}
			\ric_g=c\cdot g+\mathcal{L}_Xg, 
		\end{equation}
		where $\mathcal{L}_X$ is the Lie derivative along $X$. Especially, in case $\mathcal{L}_X=0$, the metric is called a {\it Einstein} metric. It is called {\it shrinking} when $c>0$, {\it steady} when $c=0$, and {\it expanding} when $c<0$. In particular, Lie groups, which are manifolds that possess a group structure, are valuable objects for investigating the existence of special metrics on them. Moreover, we focus particularly on the relationship between the existence of left-invariant metric on Lie groups and the algebraic properties of those Lie groups. The following theorem, which symbolizes this problem, has recently been proved. 
		\begin{Thm}[B\"ohm-Lafuente (\cite{BL})]
			If a homogeneous manifold $M$ admits an expanding Ricci soliton $g$, then $(M, g)$ is isometric to a solvable Lie group with a left-invariant metric. 
		\end{Thm}
		The statement of this theorem is known as the generalized Alekseevskii conjecture and suggests that Ricci soliton homogeneous manifolds are reducible to solvable Lie groups. However it is generally unclear which solvable Lie groups admit left-invariant Einstein and Ricci solitons. Based on the above background, we are particularly interested in left-invariant Ricci solitons on solvable Lie groups. 
				
		In the study of left-invariant Ricci solitons on Lie groups, algebraic Ricci solitons hold significant meaning. Let $(\g, \langle, \rangle)$ be a metric Lie algebra and $\Ric: \g\to\g$ be the Ricci curvature. Then, $(\g, \langle, \rangle)$ is called an {\it algebraic Ricci soliton} if there exist $c\in \mathbb{R}$ and a derivation $D$ on $\g$ such that 
		\[
			\Ric=c\cdot \id+D. 
		\] 
		It is known that if a metric Lie algebra is an algebraic Ricci soliton, then the corresponding simply-connected Lie group with left-invariant metric is a Ricci soliton. Conversely, if a Lie group is nilpotent and admits a left-invariant Ricci soliton, then the corresponding Lie algebra with an inner product is an algebraic Ricci soliton (\cite{L1}, \cite{LL}). Solvable and nilpotent  Lie algebras are closely related in the context of Ricci solitons (\cite{L2}). 
		In particular, if a nilpotent Lie algebra admits an algebraic Ricci soliton, it is possible to obtain a solvable Lie algebra that admits an algebraic Ricci soliton by performing an appropriate extension. Additionally, under certain conditions, the algebraic Ricci soliton obtained by the extension is Einstein. The detailed statements will be mentioned in Section \ref{Thm1}.

		In our previous work, we defined nilpotent Lie algebras from finite quivers without cycles. These Lie algebras are constructed by taking the commutator product of the path algebras. Also, we proved that nilpotent Lie algebras obtained by this method always admit algebraic Ricci solitons (\cite{MT}). In this paper, we define a new method to obtain solvable Lie algebras by adding the vertices as paths of length zero in a similar manner. We consider left-invariant Ricci solitons on the simply-connected Lie groups corresponding to these solvable Lie algebras. In this paper, we prove two theorems about geometric structures on solvable Lie algebras obtained by quivers.
		\begin{Thm}[Theorem 1]
			Let $Q$ be a finite quiver without cycles, and $\s_Q$ be the solvable Lie algebra obtained by $Q$. Then, the simply-connected Lie group corresponding to $\s_Q$ admits a left-invariant Ricci soliton metric. 
		\end{Thm}
		
		Moreover, we also consider whether the Ricci soliton metric obtained in Theorem 1 satisfies some additional property.  Note that the obtained metric is a Ricci soliton metric, but not an Einstein metric. There is a special type of Ricci soliton metric such as gradient and rigid. In case $X=\nabla f$ in Equation \ref{ric}, the equation can be written as 
		\[
			\ric_g+\mathrm{Hess}f=c\cdot g
		\]
		and the Reimannian metric $g$ is called a {\it gradient Ricci soliton}. It is important to consider gradient Ricci solitons since compact Ricci soliton manifolds are always gradient. Also, a gradient Ricci soliton is said to be {\it rigid} if it is isometric to a quotient of $N\times \mathbb{R}^k$ where $N$ is an Einstein manifold. It is known that on homogeneous manifolds, gradient Ricci solitons are equivalent to rigid \cite{PW}. Therefore, it is also important to consider rigid Ricci solitons. We prove the following theorem: 
		\begin{Thm}[Theorem 2]
			Let $Q$ be an oriented multi-tree, and $\s_Q$ be the solvable Lie algebra obtained by $Q$. Then, the simply-connected Lie group corresponding to $\s_Q$ admits a rigid Ricci soliton. 
		\end{Thm}
		A quiver is called an oriented multi-tree if the quiver is a tree when ignoring orientation and multiple arrows.  These theorems imply that the existence of geometric structures on solvable Lie algebras obtained by quivers follows from the properties of these quivers. 
	\section{Solvable Lie algebras obtained by quivers}
		In our previous work, we constructed nilpotent Lie algebras from finite quivers without cycles by utilizing the concept of paths within quivers. In this section, we extend this approach by constructing solvable Lie algebras from finite quivers without cycles, by adding vertices as paths of length zero. 
		\begin{Def}
			A {\it quiver} $Q=(V, E, s, t)$ is a quadruple consisting of two sets $V$, $E$, and two maps $s, t: E\to V$. Elements of $V$ and $E$ are called {\it vertices} and {\it arrows}, respectively. For an arrow $\alpha\in E$, the vertices $s(\alpha)$ and $t(\alpha)$ are called the {\it source} and the {\it target} of $\alpha$, respectively.  
		\end{Def}
		\begin{Def}
			Let $Q=(V, E, s, t)$ be a quiver and $v$, $w\in V$. A {\it path of length $m\geq 1$ from $v$ to $w$ in $Q$} is a sequence $(v|\alpha_1, \alpha_2, \ldots, \alpha_m| w)$, where $\alpha_1, \ldots, \alpha_m\in E$ and they satisfy $t(\alpha_i)=s(\alpha_{i+1})$ for $i=1, \ldots, m-1$ and $s(\alpha_1)=v$, $t(\alpha_m)=w$. In particular, a path whose source and target coincide is called a {\it cycle}.
		\end{Def}
		\begin{Def}
			A vertex $v$ is a path of length $0$ from $v$ to $v$, denoted by $(v||v)$. 
		\end{Def}
		For the path algebra $(\s_Q, \cdot)$, we refer to \cite{ASS}. We consider the path algebra for the set of two types of paths in a quiver. For a quiver $Q$, the set of all paths of length $\geq1$ is denoted by $\Path(Q)_{\geq1}$ and the set of all paths of length $\geq0$, that is all paths, is denoted by $\Path(Q)_{\geq0}$. 
		\begin{Def}
			Let $Q=(V, E, s, t)$ be a quiver, and $\Path(Q)_{\geq0}$ be the set of all paths of length $\geq0$ in $Q$. The {\it path algebra} $\s_Q$ is defined as an $\mathbb{R}$-vector space having $\Path(Q)_{\geq0}$ as basis with the following product: it is bilinear, and the product between two paths $x=(v|\alpha_1, \ldots, \alpha_k| w)$ and $y=(v'|\beta_1, \ldots, \beta_{\ell}|w')$ is defined by the concatenation, that is
			\[
        				x\cdot y=
        				\begin{cases}
        					(v|\alpha_1, \ldots, \alpha_k, \beta_1, \ldots, \beta_{\ell}|w') & (w=v'), \\
        					0 & (\text{otherwise}).
        				\end{cases}
        			\]
			The space $\s_Q$ equipped with the bracket product $[x, y]=x\cdot y-y\cdot x$ is called the {\it Lie algebra obtained by a quiver}. 
		\end{Def}
		\begin{Rem}\label{blacket_V}
			From the above definition, for $v$, $v'\in V$, the following holds: 
			\[
        				(v||v)\cdot (v'||v')=
        				\begin{cases}
        					(v||v) & (v=v'), \\
        					0 & (v\neq v').
        				\end{cases}
        			\]
			Therefore, in both cases, $[(v||v), (v'||v')]=0$. 
		\end{Rem}
		\begin{Rem}
			Note that we will represent  paths of length $\geq1$ by omitting their sources and targets, that is
			\[
				\alpha_1\cdots\alpha_m\coloneqq(v|\alpha_1, \ldots, \alpha_m|w). 
			\]
			In addition, for paths of length 0, we simply use the vertex symbol, that is
			\[
				v\coloneqq(v||v). 
			\]
		\end{Rem}
		\begin{Prop}
			Let $Q$ be a quiver of length $m$, and $\Path(Q)_{\geq0}$ be as above. Then the obtained Lie algebra $\s_Q\coloneqq\mathrm{span}(\Path(Q)_{\geq0})$ is a solvable Lie algebra with $[\s_Q, \s_Q]=\n_Q$, where $\n_Q=\mathrm{span}(\Path(Q)_{\geq1})$ is the nilpotent Lie algebra obtained by $Q$ defined in \cite{MT}.  
		\end{Prop}
		\begin{proof}
			It satisfies $[x, y]\in \mathrm{span}(\Path(Q)_{\geq1})$ for $x$ and $y\in \Path(Q)_{\geq0}$. Therefore,  $[\s_Q, \s_Q]\subset\n_Q$ holds. Also, since $[s(x), x]=x$ for $x\in\Path(Q)_{\geq1}$, it satisfies $\n_Q\subset[\s_Q, \s_Q]$.  Therefore, since the Lie algebra $\n_Q$ is nilpotent, $\s_Q=[\n_Q \n_Q]$ is solvable.    
		\end{proof}
		We call $\n_Q$ and $\s_Q$ the {\it nilpotent and solvable Lie algebra obtained by} $Q$, respectively. 
	\section{Theorem 1}\label{Thm1}
		In this section, we show the main theorem 1, that is the solvable Lie algebra $\s_Q$ obtained by a quiver admits an algebraic Ricci soliton. First, we recall the definitions of a simple directed graph and a simple undirected graph. 
		\begin{Def}
			A {\it simple directed graph} $D\Gamma=(V, E)$ is a pair consisting of a set of vertices $V$ and a set of edges $E\subset\{(x, y)\in V\times V \mid x\neq y\}$. For $e=(x, y)\in E$, the vertices $x$ and $y$ are called the {\it source} and the {\it target} of $e$, respectively. 
		\end{Def}
		\begin{Def}
			A {\it simple undirected graph} $\Gamma=(V, E)$ is a pair consisting of a set of vertices $V$ and a set of edges $E\subset \{\{x, y\}\mid x, y\in V, x\neq y\}$. 
		\end{Def}
		Next, we recall the definition of a path in a simple undirected graph. 
		\begin{Def}
			Let $\Gamma=(V, E)$ be a simple undirected graph. A {\it path between $v_1$ and $v_m$ in $\Gamma$} is a sequence of $v_1v_2\cdots v_m$, where $v_1, \ldots, v_m\in V$ and they satisfy $\{v_i, v_{i+1}\}\in E$ for $i=1, \ldots, m-1$.  In particular, if $\{v_1, v_m\}\in E$, the path is called a {\it cycle}. 
		\end{Def}
		Next, we define an operation that ignores multiple arrows and directions in a quiver. We refer to \cite{D}. 
		\begin{Def}
			Let $Q=(V, E, s, t)$ be a quiver. Then the {\it underlying simple directed graph} $Q
			^{\mathrm{simp}}$ of $Q$ is defined by 
			\[
				Q^{\mathrm{simp}}=(V, E^{\mathrm{simp}}=\{(s(e), t(e))\mid e\in E\}). 
			\]
		\end{Def}
		\begin{Def}
			Let $D\Gamma=(V, E)$ be a simple directed graph. Then, the {\it underlying simple undirected graph} $(D\Gamma)^{\mathrm{und}}$ is defined by 
			\[
				(D\Gamma)^{\mathrm{und}}=(V, E^{\mathrm{und}}=\{\{x, y\}\mid (x, y)\in E\}). 
			\]
		\end{Def}
		\begin{Def}
			Let $\Gamma=(V, E)$ be a simple undirected graph. Then, $\Gamma$ is said to be {\it connected} if  for any $x$, $y\in V$, there exists a path between $x$ and $y$. 
		\end{Def}
		\begin{Def}
			A quiver $Q$ is said to be {\it weakly connected} if $(Q^{\mathrm{simp}})^{\mathrm{und}}$ is connected. 
		\end{Def}
		\begin{Prop}\label{center}
			Let $Q=(V, E, s, t)$ be a weakly connected finite quiver without cycles, $\s_Q$ be the solvable Lie algebra obtained by $Q$, and $z(\s_Q)$ be the center of $\s_Q$. Then, $z(\s_Q)=\mathrm{span}\{v_0\}$, where $v_0=\sum_{v\in V}v$. 
		\end{Prop}
		\begin{proof}
			First, we prove $v_0\in z(\s_Q)$, that is $[v_0, x]=0$ holds for any $x\in\Path(Q)_{\geq0}=V\sqcup\Path(Q)_{\geq1}$.  By Remark \ref{blacket_V}, for $v\in V$, it satisfies $[v_0, v]=0$. In case of $x\in\Path(Q)_{\geq1}$, it satisfies 
			\[
				[v_0, x]=\sum_{v\in V}[v, x]=[s(x), x]+[t(x), x]=x-x=0. 
			\]
			Next, we prove $z(\s_Q)\subset \mathrm{spsn}\{v_0\}$. Take $p=\sum_{v\in V}a_v v+\sum_{x\in\Path(Q)_{\geq1}}b_x x\in z(\s_Q)$. We show that it satisfies $p\in\mathrm{span}\{v_0\}$. First, we prove that $b_x=0$ for all $x\in \Path(Q)_{\geq1}$. For each $v\in V$, we have
			\[
				0=[v, p]=\sum_{s(x)=v, x\in\Path(Q)_{\geq1}}b_xx+\sum_{t(y)=v, y\in\Path(Q)_{\geq1}}b_y(-y). 
			\]
			Since $Q$ has no cycles, there exist no $x\in\Path(Q)_{\geq1}$ such that $s(x)=t(x)$. Therefore, there exists no common $x\in\Path(Q)_{\geq1}$ between the first and second sigmas. For all $x\in \Path(Q)_{\geq1}$, it satisfies $b_x=0$, because $\Path(Q)_{\geq1}$ is linearly independent. Next, we prove that $a_v$ is constant for all $v\in V$. For each $e\in E$, it satisfies
			\[
				0=[p, e]=\sum_{v\in V}a_v[v, e]=(a_{s(e)}-a_{t(e)})e. 
			\]
			Thus, it satisfies $a_{s(e)}=a_{t(e)}$. Since $Q$ is weakly connected, all vertices are connected by arrows. Therefore, for all $v$, $v'\in V$, we have $a_v=a_{v'}$. Thus, $p=av_0\in\mathrm{span}\{v_0\}$ holds. 
		\end{proof}
		
		Nilpotent Lie algebras and solvable Lie algebras are closely related in the context of algebraic Ricci solitons. If  nilpotent Lie algebras admit algebraic Ricci solitons, it is possible to obtain solvable Lie algebras that admit algebraic Ricci solitons by performing appropriate extension, which will be described here. Let $\n$ be a nilpotent Lie algebra and $\langle, \rangle$ be an inner product on $\n_Q$. Also, an abelian subalgebra $\mathfrak{a}$ of symmetric derivation of $(\n, \langle, \rangle)$ is an abelian subspace of $\{A: \n\to\n\mid A\in\Der(\n), \forall x, y\in\n, \langle Ax, y\rangle=\langle x, Ay\rangle\}$. Then, the Lie bracket on $\s=\mathfrak{a}\oplus\n$ is defined by $[A, x]=Ax$ for $A\in\mathfrak{a}$ and $x\in\n$. 
		\begin{Thm}[\cite{L2}]\label{L2}
			Let $(\n, \langle, \rangle_1)$ be an algebraic Ricci soliton nilpotent Lie algebra, say with Ricci operator 
			\[
				\Ric_{\n}=c\id_{\n}+D_1, \quad c<0, D_1\in\Der(\n), 
			\]
			and $\mathfrak{a}$ be an abelian Lie algebra of symmetric derivations of $(\n, \langle, \rangle_1)$. Then,  the solvable Lie algebra $\s=\mathfrak{a}\oplus\n$ with the inner product $\langle, \rangle$ given by
			\begin{itemize}
				\item $\langle, \rangle|_{\n\times\n}=\langle, \rangle_1$, 
				\item $\langle \mathfrak{a}, \n\rangle=0$, 
				\item $\langle v, v\rangle=-(1/c)\mathrm{tr}(\ad_v)^2 \quad \forall v\in\mathfrak{a}$, 
			\end{itemize}
			is an algebraic Ricci soliton with $\Ric=c\id+D$, where $D\in\Der(\s)$. Furthermore, $(\s, \langle, \rangle)$ is Einstein if and only if $D_1\in \mathfrak{a}$. 
		\end{Thm}
		In our previous work, we proved that the nilpotent Lie algebras obtained by quivers admit algebraic Ricci solitons. 
		\begin{Thm}[\cite{MT}]\label{MT}
			Let $Q$ be a finite quiver without cycles, and $\n_Q$ be the nilpotent Lie algebra obtained by $Q$. Then, there exists an inner product $\langle, \rangle$ on $\n_Q$ such that $\langle, \rangle$ is an algebraic Ricci soliton satisfying $\Ric=-\id+D$, where $D\in \Der(\n_Q)$ is diagonal with respect to $\Path(Q)_{\geq1}$.   
		\end{Thm}	
		We examine geometric structures on $\s_Q$ using Theorem \ref{L2}. The obtained solvable Lie algebra $\s_Q=\mathrm{span} V\oplus\n_Q$ does not satisfy conditions in Theorem \ref{L2}. In fact, $\mathrm{span} V$ is not an abelian subalgebra of symmetric derivations of $(\n_Q, \langle, \rangle)$, since it has a nonzero center (see Proposition \ref{center}). Therefore, let us construct solvable subalgebras of $\s_Q$. 
		\begin{Lem}\label{symmetric derivation}
			Let $Q=(V, E, s, t)$ be a weakly connected finite quiver without cycles, $V=\{v_1, v_2, \ldots, v_N\}$, $(\n_Q, \langle, \rangle)$ be the algebraic Ricci soliton nilpotent Lie algebra obtained by $Q$, $\s_Q$ be the solvable Lie algebra obtained by $Q$, and $\mathfrak{a}_Q\coloneqq\mathrm{span}\{v_1-v_2, v_2-v_3, \ldots, v_{N-1}-v_N\}$.  Then, the following hold. 
			\begin{enumerate}
				\item $\s_Q=\mathrm{span}\{v_0\}\oplus \s'_Q$ as a Lie algebra, where $v_0=\sum_{v\in V}v$ and $\s'_Q\coloneqq\mathfrak{a}_Q\oplus\n_Q$, 
				\item $\ad_{\mathfrak{a}_Q}=\{\ad_A|_{\n_Q}\mid A\in \mathfrak{a}_Q\}$ is an abelian subalgebra of symmetric derivations of $(\n_Q, \langle, \rangle)$,  
				\item $\s'_Q=\mathfrak{a}_Q\oplus\n_Q$ is isomorphic to $\ad_{\mathfrak{a}_Q}\oplus\n_Q$. 
			\end{enumerate} 
		\end{Lem}
		\begin{proof}
			First of all, we prove (1). By Lemma \ref{center}, we have $\mathrm{span}\{v_0\}=z(\s_Q)$. Also, $[\s'_Q, \s'_Q]=\n_Q\subset\s'_Q$ holds. Therefore, $\s_Q=\mathrm{span}\{v_0\}\oplus\s'_Q$ is direct sum as a Lie algebra. 
			
			Next,  we prove (2). $\Path(Q)_{\geq1}$ is an orthogonal basis of $\n_Q$ with respect to $\langle, \rangle$ and $\ad_A|_{\n_Q}$ is diagonal with respect to $\Path(Q)_{\geq1}$ for $A\in\mathfrak{a}_Q$. Therefore, $\ad_{\mathfrak{a}_Q}$ is symmetric and abelian. 
			Thus, $\ad_{\mathfrak{a}_Q}$ is an abelian subalgebra of symmetric derivations of $(\n_Q, \langle, \rangle)$. 
			
			Finally, we prove (3). Let $\ad\oplus\id: \mathfrak{a}_Q\oplus\n_Q\to\ad_{\mathfrak{a}_Q}\oplus\n_Q$ be a map given by $(A, x)\mapsto(ad_A, x)$ for $A\in\mathfrak{a}_Q$ and $x\in \n_Q$. Then, $\mathrm{Ker}(\ad\oplus\id)=\mathfrak{a}_Q\cap z(\s_Q)=0$, since $z(\s_Q)=\mathrm{span}\{v_0\}$ by Lemma \ref{center}. Therefore,  $\ad\oplus\id$ is an isomorphism. Thus, $\s'_Q=\mathfrak{a}_Q\oplus\n_Q$ is isomorphic to $\ad_{\mathfrak{a}_Q}\oplus\n_Q$. 
		\end{proof}
		Finally, we show the main theorem 1 by applying Theorems \ref{L2} and \ref{MT}. 
		\begin{Thm}[Theorem 1]\label{sol_ARS}
			Let $Q$ be a weakly connected finite quiver without cycles, and $\s_Q$ be  the solvable Lie algebra obtained by $Q$. Then, there exists an inner product $\langle, \rangle$ on $\s_Q$ such that $(\s_Q, \langle, \rangle)$ is algebraic Ricci soliton. 
		\end{Thm}
		\begin{proof}
			Form Theorem \ref{MT}, for a finite quiver $Q=(V, E, s, t)$ without cycles, there exists an inner product $\langle, \rangle_1$ on $\n_Q$ such that $(\n_Q, \langle, \rangle_1)$ is algebraic Ricci soliton with $\Ric_{\n_Q}=-\id_{\n_Q}+D_1$, where $D_1\in \Der(\n_Q)$. This inner product $\langle, \rangle$ on $\s'_Q=\mathfrak{a}_Q\oplus\n_Q$ defined in Lemma \ref{symmetric derivation} is given by
			\begin{itemize}
				\item $\langle, \rangle|_{\n_Q\times\n_Q}=\langle, \rangle_1$, 
				\item $\langle A, \n_Q\rangle=0\quad \forall A\in \mathfrak{a}_Q$, 
				\item $\langle A, A\rangle=\mathrm{tr}(\ad_A)^2 \quad \forall A\in \mathfrak{a}_Q$. 
			\end{itemize} 
			Therefore, by Lemma \ref{symmetric derivation}, $(\s'_Q, \langle, \rangle)$ satisfies the assumptions of Theorem \ref{L2}. Thus, by Theorem \ref{L2}, the solvable Lie algebra $(\s'_Q, \langle, \rangle)$ is algebraic Ricci soliton. Since $\mathrm{span}\{v_0\}$ is abelian and $\s_Q=\mathrm{span}\{v_0\}\oplus\s'_Q$ , the solvable Lie algebra $\s_Q$ admits algebraic Ricci soliton. 
		\end{proof}
	\section{Theorem 2}
		 In this section, we show the main theorem 2. We have shown that the solvable Lie algebra $\s_Q$ obtained by $Q$ admits an algebraic Ricci soliton in Theorem \ref{sol_ARS}. If a quiver is a tree when ignoring orientation and multiple arrows, the quiver is said to be an oriented multi-tree. We prove that if a quiver is an oriented multi-tree, then the inner product obtained in Theorem \ref{sol_ARS} is rigid, that is a metric isometric to the direct sum of a flat metric and an Einstein metric. 
		 
		 In Lemma \ref{symmetric derivation}, we defined $\s'_Q=\mathfrak{a}_Q\oplus\n_Q$. We will show that $\s'_Q$ is Einstein with respect to $\langle, \rangle|_{\s'_Q\times\s'_Q}$, the restriction of $\langle, \rangle$ defined in Theorem \ref{sol_ARS}. From Theorem \ref{L2}, $(\s'_Q, \langle, \rangle|_{\s'_Q\times\s'_Q})$ is Einstein if and only if $D\in \ad_{\mathfrak{a}_Q}$, where $\Ric_{\n_Q}=c\cdot\id+D$. First, the following lemma states the conditions for a derivation to belong to $\ad_{\mathfrak{a}_Q}$. 
		\begin{Lem}\label{EPath}
			Let $Q=(V, E, s, t)$ be a weakly connected finite quiver without cycles, $\n_Q$ be the nilpotent Lie algebra obtained by $Q$, $\Path(Q)_{\geq1}$ be a set of all paths of length$\geq1$ in $Q$, and $D$ be a derivation of $\n_Q$. Then the following are equivalent: 
			\begin{enumerate}[(1)]
				\item $D\in \ad_{\mathfrak{a}_Q}$. 
				\item there exist $\lambda: V\to\mathbb{R}$ such that for $x\in \Path(Q)_{\geq1}$, $Dx=(\lambda_{s(x)}-\lambda_{t(x)})x$, and $\sum_{v\in V}\lambda_v=0$. 
				\item there exist $\lambda: V\to\mathbb{R}$ such that for $e\in E$, $De=(\lambda_{s(e)}-\lambda_{t(e)})e$, and $\sum_{v\in V}\lambda_v=0$. 
			\end{enumerate}
		\end{Lem}
		\begin{proof}
			We first show $(1) \Rightarrow (2)$. By (1), there exists $a\in\mathfrak{a}_{Q}$ such that $D=\ad_a$. From $a\in\mathfrak{a}_Q$, for each $v\in V$, there exists $\lambda_v\in\mathbb{R}$ such that $a=\sum_{v\in V}\lambda_v v$ and $\sum_{v\in V}\lambda_v=0$, that is there exists a map $\lambda: V\to\mathbb{R}$ such that $\sum_{v\in V}\lambda_v=0$. For any $x\in \Path(Q)_{\geq1}$, it is easy to show
			\[
				Dx=\ad_ax=\sum_{v\in V}\lambda_v\ad_vx=(\lambda_{s(x)}-\lambda_{t(x)})x. 
			\]
			Therefore, (2) holds. Next, we show (2) $\Rightarrow$ (1). Assume that (2) holds. For a map $\lambda: V\to\mathbb{R}$, define $a=\sum_{v\in V}\lambda_v v$. From $\sum_{v\in V}\lambda_v=0$, it satisfies $a\in\mathfrak{a}_{Q}$. For $x\in\Path(Q)_{\geq1}$, it satisfies 
			\[
				\ad_ax=\sum_{v\in V}\lambda_v\ad_vx=(\lambda_{s(x)}-\lambda_{t(x)})x=Dx. 
			\]
			Since $\Path(Q)_{\geq1}$ is a basis of $\n_Q$, it satisfies $D=\ad_a$. 
			
			 It is clear that (3) follows from (2). We show that (3) implies (2). Assume that (3) holds. We show (2) using induction on the length of paths. From (3), the claim holds when the length of paths is 1. Assume that the claim holds when the length of paths is at most $n-1$. We show that the claim holds for paths of length $n$. Take a path $x=x'e$ of length $n$, where $x'$ and $e$ are paths of length $n-1$ and $1$, respectively. From the assumption of induction, it satisfies $Dx'=(\lambda_{s(x')}-\lambda_{t(x')})x'$ and $De=(\lambda_{s(e)}-\lambda_{t(e)})e$. Therefore, it satisfies
			 \begin{align*}
			 	Dx&=[Dx', e]+[x', De] & (\because D\in\Der(\n_Q)) \\
				&=(\lambda_{s(x')}-\lambda_{t(x')})x'e+(\lambda_{s(e)}-\lambda_{t(e)})x'e \\
				&=(\lambda_{s(x)}-\lambda_{t(x)})x &(\because t(x')=s(e)). 
			 \end{align*}
			Thus, by induction, (2) holds. 
		\end{proof}
		
		To prove Theorem 2, we have only to construct a map $\lambda$ as above, for which we will use the properties of a tree. Therefore we recall the definitions of a tree and some properties. 
		\begin{Def}
			 Let $\Gamma=(V, E)$ be a simple undirected graph. Then, $\Gamma$ is called a {\it tree} if $\Gamma$ is connected and does not have cycles. 
		\end{Def}
		\begin{Def}
			Let $D\Gamma=(V, E)$ be a simple directed graph. Then, $D\Gamma$ is called an {\it oriented tree} if the simple undirected graph $(D\Gamma)^{\mathrm{und}}$ is a tree. 
		\end{Def}
				\begin{Def}
			Let $\Gamma=(V, E)$ be a simple undirected graph, and $v\in V$ be a vertex. Then, the {\it degree} of $v$ is defined by 
			\[
				\mathrm{deg}(v)=\#\{e \in E\mid v\in e\}. 
			\]
			A vertex $v$ is called a {\it leaf vertex} if $\mathrm{deg}(v)=1$. 
		\end{Def}
		\begin{Prop}[\cite{D}]\label{treeprop1}
			Let $T=(V, E)$ be a tree. Then, the following properties hold: 
			\begin{enumerate}[(1)]
				\item It satisfies $\#V=\#E+1$.  
				\item $T$ has a leaf vertex. 
			\end{enumerate}
		\end{Prop}
		The number of vertices in a tree $T$ is called the {\it order of} $T$. 
		Let $\Gamma=(V, E)$ be a simple undirected graph, and $W\subsetneq V$ be a proper subset of $V$. Then, a simple undirected graph $\Gamma-W$ is defined by
		\[
			\Gamma-W=(V\setminus W, E\setminus\{e\in E\mid \exists w\in W \text{ s.t. }w\in e\}). 
		\]
		\begin{Prop}[\cite{D}]\label{treeprop2}
			Let $T$ be a tree of order $n\geq2$, and $v\in V$ be a leaf vertex. Then, $T-v$ is a tree of order $n-1$. 
		\end{Prop}
		Next, using properties of a tree, we prove the key lemma of the main theorem 2. Let $X$ be a vector space, and $A$ be a linear transformation on $X$ which is diagonal with respect to a basis $\mathcal{B}$. Then we denote the eigenvalue of $A$ at $e\in\mathcal{B}$ as $\mu(A, e)$, that is $Ae=\mu(A, e)e$. 
		\begin{Lem}\label{treelem}
			Let $DT=(V, E)$ be an oriented tree of order $n\geq2$, and $A$ be a linear transformation on $\mathrm{span~}E$ which is diagonal with respect to the basis $E$. Then, there exists $\lambda: V\to \mathbb{R}$ such that for any $e=(x, y)\in E$, $Ae=(\lambda_x-\lambda_y)e$ and $\sum_{v\in V}\lambda_v=0$. 
		\end{Lem}
		\begin{proof}
			We prove the claim by induction with respect to the order of a tree. In case of the order $n=2$, it satisfies $V=\{x, y\}$ and $E=\{e=(x, y)\}$. Since $A$ is a matrix of size $1$, $A=(\mu(A, e))$ holds. For $e=(x, y)\in E$, we define
			\[
				\lambda_{x}=\frac{1}{2}\mu(A, e), \quad \lambda_{y}=-\frac{1}{2}\mu(A, e). 
			\]
			Then, it satisfies $(\lambda_{x}-\lambda_{y})e=Ae$ and $\lambda_{x}+\lambda_{y}=0$. 
			
			Next, we prove that a tree of order $n$ satisfies the claim, assuming a tree of order $n-1$ satisfies the claim. Let $T$ be a tree of order $n$. By Propositions \ref{treeprop1}, \ref{treeprop2}, there exists a leaf vertex $v'$ such that $T-v'$ is a tree of order $n-1$,  and let $e'$ be the unique edge connected to $v'$. By the assumption of induction, there exists $\lambda': V\setminus \{v'\}\to\mathbb{R}$ such that for any $e=(x, y)\in E\setminus \{e'\}$, it satisfies $Ae=(\lambda'_x-\lambda'_y)e$ and $\sum_{v\in V\setminus\{v'\}}\lambda'_{v}=0$. We define $\overline{\lambda'}: V\to \mathbb{R}$ as follows: $\overline{\lambda'}|_{V\setminus\{v'\}}=\lambda'$ and  
			\[
				\overline{\lambda'}_{v'}=
				\begin{cases}
					\mu(A, e')+\lambda'_{w} &(e'=(v', w)), \\
					-\mu(A, e')+\lambda'_{u} &(e'=(u, v')). 
				\end{cases}
			\]
			We prove that $\overline{\lambda'}$ is compatible with $A$, that is $Ae=(\overline{\lambda'}_x-\overline{\lambda'}_y)e$ holds for  any $e=(x, y)\in E$. 

			Since $\overline{\lambda'}$ is compatible with $A$ for $e=(x, y)\in E\setminus \{e'\}$, we have only check it for $e'$. In case of $e'=(v', w)$, it satisfies 
			\[
				Ae'=\mu(A, e')e'= (\overline{\lambda'}_v'-\lambda'_w)e'=(\overline{\lambda'}_v'-\overline{\lambda'}_w)e'. 
			\]
			The same holds in case of $e'=(u, v')$. Define $\lambda: V\to \mathbb{R}$ by $\lambda_v=\overline{\lambda'}_v-\frac{1}{n}\sum_{w\in V}\overline{\lambda'}_w$ for $v\in V$. Then the map $\lambda$ is still compatible with $A$. Moreover, $\lambda$ satisfies $\sum_{v\in V}\lambda_v$=0.  
		\end{proof}
		\begin{Ex}
			The proof of Lemma \ref{treelem} is illustrated by the following example. The following gives an oriented tree $DT$ of order $4$: $V=\{v_1, v_2, v_3, v_4\}$, $E=\{a=(v_1, v_2), b=(v_2, v_3), c=(v_2, v_4)\}$. It is illustrated in Figure \ref{tree1}. Let $A$ be a linear transformation on $\mathrm{span~}E$ and $P=\mathrm{diag}(p_1, p_2, p_3)$ be the representation matrix of $A$ with respect to $E$. 
			\begin{figure}[H]
				\includegraphics[scale=0.4]{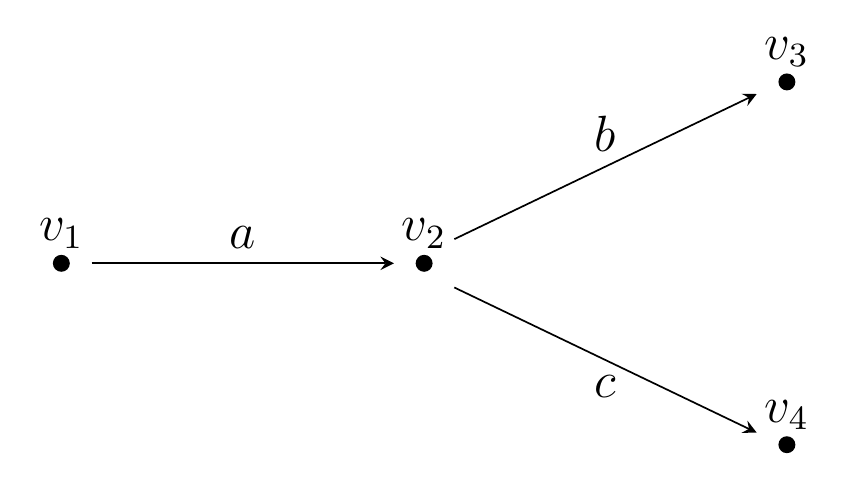}
				\caption{The oriented tree $DT$}
				\label{tree1}
			\end{figure}
			At this time, we will construct a map $\lambda: V\to \mathbb{R}$. Since $v_1$ is a leaf vertex of $DT$, $(DT)'\coloneqq DT-v_1=(V', E')$ is an oriented tree of order $3$, where $V'=\{v_2, v_3, v_4\}$ and $E'=\{b, c\}$. It is illustrated as in Figure \ref{tree2}. 
			\begin{figure}[H]
				\includegraphics[scale=0.4]{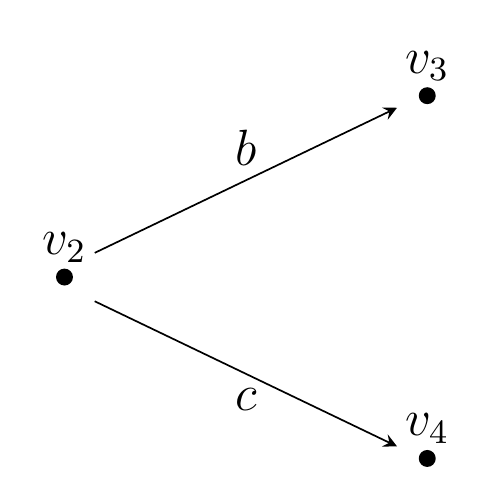}
				\caption{The oriented tree $(DT)'$}
				\label{tree2}
			\end{figure}
			Next since $v_3$ is a leaf vertex of $(DT)'$, $(DT)''\coloneqq(DT)'-v_3=(V'', E'')$ is an oriented tree of order $2$, where $V''=\{v_2, v_4\}$, $E''=\{c\}$. It is illustrated as in Figure \ref{tree3}. 
			\begin{figure}[H]
				\includegraphics[scale=0.4]{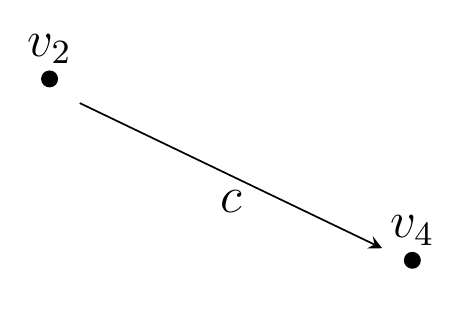}
				\caption{The oriented tree $(DT)''$}
				\label{tree3}
			\end{figure}
			We determine a map $\lambda: V\to \mathbb{R}$. First of all, a map $\lambda'': V''\to\mathbb{R}$ is given by
			\begin{itemize}
				\item $\lambda''_{v_2}=(1/2)\mu(P, c)=(1/2)p_3$, 
				\item $\lambda''_{v_4}=-(1/2)\mu(P, c)=-(1/2)p_3$. 
			\end{itemize}
			Next, by extending $\lambda''$ to $V'$, we define $\overline{\lambda''}: V'\to\mathbb{R}$ by
			\begin{itemize}
				\item $\overline{\lambda''}|_{V''}=\lambda''$, 
				\item $\overline{\lambda''}_{v_3}=\lambda''_{v_2}-\mu(P, b)=(1/2)p_3-p_2$. 
			\end{itemize}
			Define $\lambda': V'\to\mathbb{R}$ by $\lambda'_{v_i}=\overline{\lambda''}_{v_i}-\frac{1}{3}\sum_{j=2}^4\overline{\lambda''}_{v_j}$ for $i=2, 3, 4$. Thus, $\lambda'_{v_2}=(1/3)(p_2+p_3)$, $\lambda'_{v_3}=(1/3)(-2p_2+p_3)
			$ and $\lambda'_{v_4}=(1/3)(p_2-2p_3)$. 
			Next,  by extending $\lambda'$ to $V$, we define $\overline{\lambda'}: V\to\mathbb{R}$ by 
			\begin{itemize}
				\item $\overline{\lambda'}|_{V'}=\lambda'$, 
				\item $\overline{\lambda'}_{v_1}=\lambda'_{v_2}+\mu(P, a)=(1/3)(3p_1+p_2+p_3)$. 
			\end{itemize}
			Finally, $\lambda: V\to \mathbb{R}$ is given by $\lambda_{v_i}=\overline{\lambda'}_{v_i}-\sum_{j=1}^4\overline{\lambda'}_{v_j}$. Therefore, $\lambda_{v_1}=(1/4)(3p_1+p_2+p_3)$, $\lambda_{v_2}=(1/4)(-p_1+p_2+p_3)$, $\lambda_{v_3}=(1/4)(-p_1-3p_2+p_3)$ and $\lambda_{v_4}=(1/4)(-p_1+p_2-3p_3)$. 
		\end{Ex}
		\begin{Def}
			Let $Q$ be a quiver. Then, $Q$ is called an {\it oriented multi-tree} if $(Q^{\mathrm{simp}})^{\mathrm{und}}$ is a tree. 
		\end{Def}
		\begin{Thm}[Theorem 2]
			Let $Q=(V, E, s, t)$ be an oriented multi-tree, and $\s_Q$ be the solvable Lie algebra obtained by $Q$. Then, $\s_Q$ admits an inner product which is isometric to a flat inner product and an Einstein inner product. 
		\end{Thm}
		\begin{proof}
			By Theorem \ref{MT}, the nilpotent Lie algebra $\n_Q$ obtained by quiver $Q$ admits algebraic Ricci soliton with $\Ric_{\n_Q}=-\id+D$, where $D\in\Der(\n_Q)$. We use the notation $\s'_Q$, $\mathfrak{a}_Q$ , $\ad_{\mathfrak{a}_Q}$ and $v_0$ defined in Lemma \ref{symmetric derivation}. Recall that $\s_Q=\mathrm{span}\{v_0\}\oplus\s'_Q$. By Theorem \ref{sol_ARS}, there exists an inner product $\langle, \rangle$ on $\s'_Q$ such that $(\s'_Q, \langle, \rangle)$ is algebraic Ricci soliton. We prove that $\langle, \rangle$ is an Einstein metric on $\s'_Q$. Since $(\s'_Q, \langle, \rangle)$ satisfies the assumption of Theorem \ref{L2},  we have only to prove $D\in \ad_{\mathfrak{a}_Q}$. Since $Q$ is an oriented multi-tree, $Q^{\mathrm{simp}}$ is weakly connected. Also, the derivation $D$ is diagonal with respect to $\Path(Q)_{\geq1}$. By Lemma \ref{treelem}, there exists $\lambda: V\to \mathbb{R}$ such that for $e\in E^{\mathrm{simp}}$, $De=(\lambda_{s(e)}-\lambda_{t(e)})e$. It satisfies $\mu(D, e)=\mu(D, e')$ if $s(e)=s(e')$ and $t(e)=t(e')$ hold for $e, e'\in E$. Therefore, for $e\in E$, the clime holds true. Moreover, by Lemma \ref{EPath},  it satisfies $D\in\ad_{\mathfrak{a}_Q}$. 
		\end{proof}
		\begin{Ex}
			Consider the following quiver $Q$ in Figure \ref{nottree2}. Let $\s'_Q$ be the solvable Lie algebra defined in Lemma \ref{symmetric derivation}. Then, we will show that $\s'_Q$ does not admit an Einstein metric. 
			\begin{figure}[H]
				\includegraphics[scale=0.4]{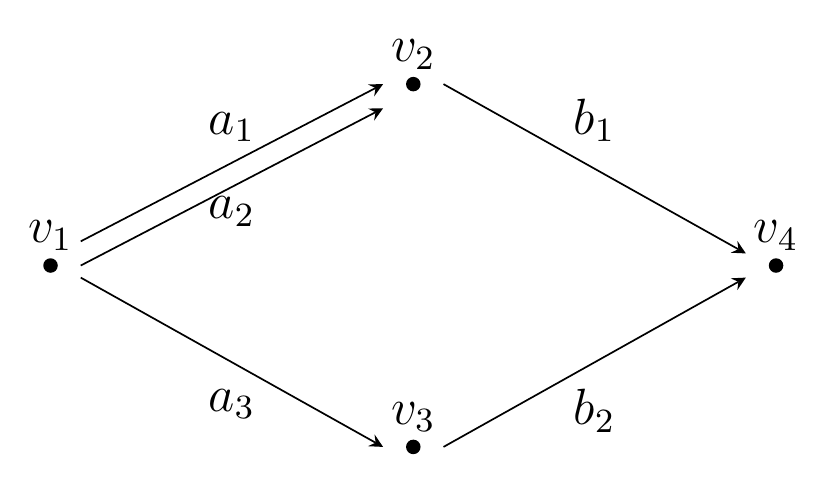}
				\caption{not oriented multi-tree $Q$}
				\label{nottree2}
			\end{figure}
			Since $(Q^{\mathrm{simp}})^{\mathrm{und}}$ is not a tree, a quiver $Q$ is not an oriented multi-tree. Let $\n_Q$ be the nilpotent Lie algebra obtained by $Q$, and $\s'_Q$ be a solvable Lie algebra defined in Lemma \ref{symmetric derivation}. Since $Q=(V, E, s, t)$ is a finite quiver without cycles, the nilpotent Lie algebra $\n_Q$ obtained by $Q$ admits algebraic Ricci soliton ($\Ric_{\n_Q}=-\id+D$). First, using the method in \cite{MT}, the derivation $D$ with respect to a basis $\Path(Q)_{\geq1}=\{a_1, a_2, a_3, b_1, b_2, a_1b_1, a_2b_1, a_3b_2\}$ is given by 
			\[
				D=\mathrm{diag}\left(\frac{3}{4}, \frac{3}{4}, \frac{2}{3}, \frac{2}{4}, \frac{2}{3}, \frac{5}{4}, \frac{5}{4}, \frac{4}{3}\right)	. 
			\]			
			By Theorem \ref{sol_ARS}, there exists an inner product $\langle, \rangle$ on $\s'_Q$ such that $(\s'_Q, \langle, \rangle)$ is an algebraic Ricci soliton. Assume that $\s'_Q$ admits an Einstein metric. Since an algebraic Ricci soliton metric is unique up to  isometry and scaling, $(\s'_Q, \langle, \rangle)$ is an Einstein. By Theorem \ref{L2}, $\langle, \rangle$ is an Einstein metric if and only if $D\in\ad_{\mathfrak{a}_Q}$, that is there exists a map $\lambda: V\to \mathbb{R}$ such that $Dx=(\lambda_{s(x)}-\lambda_{t(x)})x$ for any $x\in\Path(Q)_{\geq1}$. 
			From the assumption, the following hold;
			\begin{align*}
				\lambda_{v_1}-\lambda_{v_4}&=\lambda_{s(a_1b_1)}-\lambda_{t(a_1b_1)}=\mu(D, a_1b_1)=\frac{5}{4},  \\
				\lambda_{v_1}-\lambda_{v_4}&=\lambda_{s(a_3b_2)}-\lambda_{t(a_3b_2)}=\mu(D, a_3b_2)=\frac{4}{3}. 
			\end{align*}
			This is contradictory, which proves $\langle, \rangle$ is not an Einstein metric. Therefore, $\s'_Q$ do not admit Einstein metrics. 			
		\end{Ex}
	\section*{Acknowledgements}
		I would like to express my sincere gratitude to Hiroshi Tamaru for his guidance and support during this research. I also appreciate the cooperation of my colleagues in Tamaru's research group. 
			
\end{document}